\documentclass[12pt,a4paper]{article}
\usepackage[utf8]{inputenc}
\usepackage{amsmath}
\usepackage{amsfonts}
\usepackage{amssymb}
\usepackage{graphicx}
\usepackage{mathabx}

    \usepackage{pdfsync}

\newtheorem{theorem}{Theorem}[section]
\newtheorem{lemma}{Lemma}[section]
\newtheorem{corollary}{Corollary}[section]

\newtheorem{remark}{Remark}[section]
\newtheorem{definition}{Definition}[section]
\newenvironment{proof}[1][Proof]{\textbf{#1.} }{\ \rule{0.5em}{0.5em}}
\author{Yves Le Jan}
\title{Connections and loops intertwinning } 
\begin{document}
\maketitle


%
 On a finite graph, we prove that trace of holonomies determine an intertwining relation between merge-and-split generators on collections of geodesic loops ensembles and Casimir operators on moduli of unitary connections. By adding a deformation part to the generator on loops, this result is extended to the Casimir operator modified in order to be self adjoint with respect to Yang-Mills measure.

\section{Generalities}\label{geol}
Recall a few usual definitions and properties for which we can refer to \cite{stfl}.\\
We consider a finite connected graph $\mathcal{G}=(X,E)$. $X$ is the set of vertices and $E$ the set of edges. We assume there are no loop edges nor multiple edges, and denote by $E^o$ the set of oriented edges.\\
A path is a sequence of vertices such that consecutive elements define an oriented edge. A based loop is a path whose starting point coincides with its endpoint (and is defined as the base point).\\ A geodesic path is a non backtracking path.
Given $\omega$ any finite path in $X$, the reduced path
$\omega^{R}$ is the geodesic arc defined by projection of the geodesic arc between the starting point and the endpoint of any lift
of $\omega$ to the universal cover.

Tree-contour-like based loops can be defined as discrete based loops whose
lifts to the universal cover of the graph are still based loops. The reduced
path $\omega^{R}$\ can equivalently be obtained by removing all
tree-contour-like based loops imbedded into it. This can be done iteratively by removing pairs of opposite consecutive edges. \\
Loops are defined as equivalence classes of based loops under the
natural shift $\theta$ defined as follows:
\begin{center}
 If  $\xi=(\xi
_{0}, \xi_1, ...,\xi_{p(\xi)}=\xi_0)$,
 $\theta\xi=(\xi
_{1},...,\xi_{p(\xi)},\xi_{p(\xi)+1}=\xi_1)$.\\
\end{center}
Given a loop $l$, there is a canonical geodesic loop $l^R$ associated with it. It is
obtained  by removing all tree-contour-like based loops imbedded into it.\\

Geodesic loops are in bijection with the set of conjugacy classes
of the fundamental group of the graph. \\
Among geodesic loops, we can distinguish specific types: Primitive geodesic loops, which are in bijection with  primitive conjugacy classes and
circuits in which the same edge is not visited twice.\\

By collection, \index{collection}we mean that each element of it has a multiplicity (which is one in the case of a set). One could alternatively refer to collections as multisets, (or non-negative integer valued point measures). \\

\section{Split and merge generators}\label{SM}
The simplest variables defined on discrete or continuous loops are the number of crossings of an oriented edge $(x,y)$%
\[
N_{x,y}(l)=\#\{i:\xi_{i}=x,\xi_{i+1}=y\}
\]
(recall the convention $\xi_{p}=\xi_{0})$ and the number of visits to a vertex $x$ %
\[
N_{x}(l)=\sum_{y}N_{x,y}(l).%
\]
Note that $N_{x}=\#\{i\geq1:\xi_{i}=x\}$\ (except for trivial one point loops
for which it vanishes, and which are not considered here).\\
 $\{ N_{x,y}(l),\, (x,y)\in E^o\}$ can be referred to as the (oriented) edge occupation field defined by the loop. We also define $N_{\{x,y\}}(l)= N_{x,y}(l)+ N_{y,x}(l)$.\\
 These definitions extend trivially to a collection $\mathcal L$ of loops. For example, $N_{x,y}(\mathcal L)=\sum_{l\in \mathcal L} N_{x,y}(l)$.\\

We define a \emph{network} to be a  $\mathbb{N}$ valued function defined on oriented edges of the graph. It is given by a matrix $k$ indexed by $X$ with  $\mathbb{N}$-valued coefficients which vanishes on the diagonal and on entries $(x,y)$ such that $\{x,y\}$ is not an edge of the graph. \emph{We say that $k$ is Eulerian iff $$ \sum_y k_{x,y}= \sum_y k_{y,x}.$$ For any Eulerian network $k$, we define $k_x$ to be $\sum_y k_{x,y}=\sum_y k_{y,x}$ and set $\vert k \vert = \sum_{x} k_{x}$.}
It is obvious that the edge occupation field $N(\mathcal{L})$) defines a random network which satisfies the Eulerian property.\\

\begin{definition}We say that a Eulerian network $j$ is a flow iff $j_{x,y}\,j_{y,x}=0$ for all edges $\{x,y\}$. \index{flow}
\end{definition}
Note that  a flow defines an orientation on edges on which it does not vanish. \\
It is easy to check that the measure $j_x=\sum_y j_{x,y}$ is preserved by the associated Markovian matrix $q[j]$ defined as follows:
$q[j]_{x,y} =\frac{j_{x,y}}{j_x}$ if $j_x >0$, $q[j]_{x,y} =\delta_{x,y}$ if $j_x =0$.\\
We can define the flow $J(k)$ associated to any Eulerian network $k$ by $$J(k)_{x,y}=1_{\{k_{x,y}-k_{y,x}>0\}}[k_{x,y}-k_{y,x}].$$


 For any vertex $x$, a Markov chain on the collections of discrete loops can then be naturally constructed as follows: two indices are chosen uniformly and independently  among the $n$ $x$-visit indices of all loops. If they belong to different loops, these loops are concatenated at these indices.  If they are distinct but on the same loop, the loop is cut at these indices to form two loops by connecting the two ends of each part. There is no change when the two choices coincide. This split and merge mechanism is inspired by \cite{pitpi}.\\
 The transition matrix of this Markov chain on collections of discrete loops will be denoted by $Q_{(x)}$ in the following.\\
%
 If we let $x$ vary, the corresponding Markov chains can be combined into a continuous time Markov chain using independent exponential jump times of rates $N_x^2$. The recurrence classes of the resulting stationary 
 continuous time chain on collections of discrete loops are the collections which induce the same network. The infinitesimal generator of this split-and-merge process on collections of discrete loops  is: 
  $$B^{SM}=\sum_x N^2_x [Q_{(x)}-I].$$ 
 A similar construction can be done for an oriented edge $(x,y)$. Two visit indices are chosen uniformly among those at which starts a $(x,y)$-crossings in some loop. If they belong to different loops, these loops are merged as described above. If they are distinct but on the same loop, the loop is disconnected in two paths from $x$ to $x$, both starting with the edge $(x,y)$, to form two loops. Finally, nothing happens if the two choices coincide. The corresponding transition matrix on collections of discrete loops is denoted by $K^+_{(x,y)}$. This transition matrix will be used on collections of geodesic loops.  It is clear that for any oriented edge $(x,y)$, the operations of splitting and merging we just defined to construct $K^+_{x,y}$ preserve geodesic loops.\\
If we let $(x,y)$ vary, they can be combined with jump rates $N_{x,y}^2$ to get a continuous time Markov chain on collections of geodesic loops. The infinitesimal generator of this split-and-merge process on collections of geodesic loops  is:             $$B^{SM+}=\sum_{x,y}N_{x,y}^2[K^+_{x,y}-I].$$ 
 We can define another split-and-merge transition matrix on collections of discrete loops which involves cancelations of opposite edge-crossings. Consider an oriented edge $(x,y)$ and a collection of loops. If both $(x,y)$ and $(y,x)$ are  traversed by at least one loop in that collection, we can define a random collection as follows: sample uniformly a random element among all loop crossings of $(x,y)$ and another one  among all loop crossings of $(y,x)$. If these crossings occur on the same loop, this loop can be decomposed into the crossing of $(x,y)$ followed by a bridge from $y$ to $y$ followed by the crossing of $(y,x)$ , followed finally by a bridge from $x$ to $x$ . These bridges induce two loops which we say are produced by a negative $\{x,y\}$-split. If these crossings occur on different loops, these loops can be decomposed respectively in a crossing of $(x,y)$ followed by a bridge from $y$ to $x$ and a crossing of $(y,x)$ followed by a bridge from $x$ to $y$. Then these two bridges can be merged in one loop.
This construction provides a transition probability on loop collections which we denote by $K^-_{\{x,y\}}$. \index{$K^-_{\{x,y\}}$}
 The merging operation may produce a loop with backtracking. Therefore we will apply the reduction map $l \longrightarrow l^R$ (Cf. section \ref{geol}) to each merging output and denote by  $K^{-\,R}_{x,y }$ the modified transition kernel. This may produce an empty collection. If we let $\{x,y\}$ vary, the corresponding Markov chains can be combined with jump rates $N_{x,y}N_{y,x} $ to generate a continuous time Markov chain. 
The infinitesimal generator of this split-and-merge process on collections of geodesic loops  is:   $$B^{SM-}=\sum_{x,y}N_{x,y}N_{y,x}[K^{-\,R}_{x,y}-I].$$
\section{Holonomies and Gauge fields}
 In this section, given a group $N$, we introduce $N$-connections on a graph and loop holonomies. 
\begin{definition}
Consider any group $N$. Define the gauge group $\mathcal{X}$ to be \index{gauge group} \index{$\mathcal{X}$} the direct product $N^X$, and let $\mathcal{C}(\mathcal G, N)$ \index{$\mathcal{C}(\mathcal G, N)$} be the set of $N$-connections i.e. maps $m$ from oriented edges into $N$ such that $m(x,y)=m(y,x)^{-1}$. An equivalence relation is naturally defined on this set: $m$ is equivalent to $m'$ if and only if there exists an element $h$ of the gauge group such that $m'(x,y)=h(y)^{-1} m({x,y})h(x)$ ($m(x,y)$ is seen as acting on the left to map a fiber above $x$ onto a fiber above $y$, as the action of the linear group on frame bundles in differentiable geometry). By definition, $N$-connection moduli are the corresponding equivalence classes. The set of  $N$-connection moduli on $\mathcal G$ is denoted by $\mathcal{C}_{\sim}(\mathcal G,N)$.  \index{connection}  \index{$\mathcal{C}_m$} \index{$\mathcal{C}_m(\mathcal G,N)$}
\end {definition}

\begin{definition}
Given any discrete loop $l_{x}$ based at $x$ and a representative $m$ of a connection modulus $A$, let $h_m(l_x)$ be the element of $N$  obtained  by multiplying the images under $m$ of the oriented edges of the loop in cyclic order, starting from the base point.  The holonomy of the corresponding loop $l$ is defined as the conjugacy class of this product. It depends only on  the connection modulus and on the loop, and will be denoted by $h_A(l)$. \index{$h_A$}  \index{holonomy}

\end {definition}
\begin{remark}\label{conrk}
a) Note that  the holonomy depends only on the geodesic loop associated with $l$.  \\
 b) Given a spanning tree rooted in $x_0$ and a connection $m\in \mathcal{C}(\mathcal {G}, N)$, define a equivalent map $\tilde{m}$ by conjugation, using the gauge group element $g_x,\, x\in X$ in which $g_x$ is the product of the values of $m$ along the unique geodesic in the tree from the root to $x$. We can check that  $\tilde{m}$ assigns the identity to all tree edges. Moreover, if $m'$ is equivalent to $m$, i.e. there exists a gauge group element $(g'_x,\, x\in X)$ such that $m'_{x,y}= [g'_x]^{-1}m_{x,y}g'_y$, then  $\tilde{m'}=[g'_{x_0}]^{-1}mg'_{x_0}$. \emph{Connection moduli are therefore in one to one correspondence with equivalence classes of $N^r$ under simultaneous conjugation.}

 \begin{remark}\label{didon}
A similar definition can be given for the gauge group. Given an element $\{g\}$ of the gauge group $\mathcal{X}$ we can associate to any discrete loop $\xi= (x_0,\, x_1, ..., x_{n-1 },\,x_0)$ the conjugacy class $q_{\{g\}}(l)$ of the product $\prod_0^{n-1}\{g\}(x_i)$.
\end{remark} 
%
%
%
 
 \end{remark}
 \emph{Example: One-forms.} One-forms on $\mathcal{G}$ are real valued functions $\omega$ on oriented
edges such that $\omega^{x,y}=-\omega^{y,x}$. They define $U(1)$ connections. For a one-form $\omega$, the $U(1)$-holonomy of a loop $l$ is given by $e^{i\int_l\omega}$. 
%
  \begin{remark}
  Assume that the graph is finite.
 If we consider all finite groups $N$, and all $N$-connection moduli, the variables $\{ h_A(l)\}$ determine the geodesic loop $l^R$. The proof, given in \cite{stfl}, relies on the fact that free groups are conjugacy separable  (\cite{Stebe}). In other terms, given two elements belonging to different conjugacy classes, there exists a finite quotient of the free group in which they are not conjugate. 
\end{remark}

An important result about holonomies is the following, in which we denote by $U(d)$ the group of unitary $d\times d$ matrices.
\begin{theorem}\label{coinc}
On a finite graph, two $U(d)$-connections are equivalent if all traces of loop holonomies coincide.
\end {theorem} 
\begin{proof}
From remark \ref{conrk} b), it is equivalent to show that given $r$ unitary matrices, the traces of all products formed with these matrices and their inverse determine these matrices up to simultaneous conjugation. This is proved in section 11 of \cite{Proc} and in \cite{TL}.
\end{proof}

\section{Random connections, and Yang-Mills measure}\label{ymil}
Let $N$ be any topological group equipped with a finite bi-invariant Haar measure. On any graph $\mathcal{G}=(X,E)$, the normalized Haar measure defines a product measure on the set of maps from $E$ to $N$. Any choice of orientation on $E$ defines a bijection between maps from oriented edges to $N$ and the group $N^E$ such that opposite edges have inverse images. Moreover the image of the product Haar measure on the group $N^E$ does not depend on the orientation choice. It is by definition the Haar measure on $\mathcal{C}( \mathcal{G},N)$. Then the Haar measure on connection moduli is defined as the image of this measure by the quotient map.\\
\emph{Let us now assume that $N$ is the unitary group $U(d)$. and that the graph is the discrete torus $[\mathbb{Z}/L\mathbb{Z}]^n$, with $n\geq 2$.} We can define the set of plaquettes $\mathcal{P}$ as the set of geodesic loops of (shortest) length 4. 
  
Consider the Haar measure on the set of $U(d)$-connection moduli on  $[\mathbb{Z}/L\mathbb{Z}]^{n}$ defined by the Haar measure on $U(d)$. For any positive constant $k$, we can assign to each connection modulus $A$ the weight $$Y_{k}(A)=e^{-k\sum_{\eta\in \mathcal{P} } [1-\frac{1}{d}Tr(h_A(\eta))]}.$$ Then the measure of density $Y_{k}$, properly normalized, is known as a Yang-Mills measure. The holonomy traces of plaquettes with opposite orientation are conjugate complex numbers, so that $Y_{k}$ is positive.
 By remark \ref{conrk} b), the set of connection moduli on $[\mathbb{Z}/L\mathbb{Z}]^{n}$ can be identified with a set of $r=L^n(n-1)+1$ elements of $U(d)$ defined up to simultaneous conjugacy. In particular, if $N=U(1)$ the set of connections can be identified with the torus $[\mathbb{R}/2\pi \mathbb{R}]^r$.

\begin{remark} 
- $[\mathbb{Z}/L\mathbb{Z}]^n$ is an Abelian cover of a graph with one vertex and $n$ loop edges.  Such "bouquet graphs" do not satisfy our standing assumptions but can be transformed into graphs which fit in our framework by adding two or more intermediate points on each loop edge. \\
There are other examples of lattices which can be treated similarly such as the toroidal cover of the tetrahedron for which the plaquettes are easily seen to be decagons.\\
- The construction of Yang-Mills measures could be extended to graphs which are not necessarily Abelian covers. Plaquettes could be defined as geodesic simple circuits which can intersect at most in one edge. 
\end{remark}
\section {Split-and-merge and Casimir operators} 
The Lie algebra of $U(d)$ is the space of antihermitian matrices . The Killing form defined by $Tr(WW^*)$ provides this space with a Hermitian structure. Given any antihermitian matrix $E$, let $\mathfrak{L}_E$ denote the left Lie derivative: For any smooth function $f$ on $U(d)$, $\mathfrak{L}_E f(g)=\lim_{\epsilon\rightarrow 0}\frac{1}{\epsilon}(f(\exp(\epsilon E)g)-f(g))$. In particular, given any fixed matrix $a$, $\mathfrak{L}_E Tr(ga)=Tr(Ega)$ and $\mathfrak{L}_E Tr(g^*a)=-Tr(gEa)$.\\ Given an orthonormal base $W_l$, the second order operator  $\mathfrak{A}=\sum_l {\mathfrak{L}_{W_l}}^2$, known as the Casimir operator, does not depend on the base and commutes with any left Lie derivative and with right or left multiplication by a group element. \index{$\mathfrak{L}_E$} \index{$\mathfrak{A}$}  \index{Casimir operator}

It operates on central functions i.e. functions invariant under conjugation. Note that on central functions, left and right Lie derivatives coincide.\\
For $d=1$, the Lie algebra is the line $\{i\omega,\, \omega\in \mathbb{R}\}$, and $\mathfrak{A}=\frac{d^2}{d\omega^2}.$\\
The Casimir operator $\mathfrak{A}$ is self adjoint with respect to the Haar measure. It generates a convolution semigroup $H_t$ of symmetric \index{$H_t$} (i.e. such that $H_t(u)=H_t(u^*)$) central functions $H_t$ (the heat semigroup) which solves the Fokker-Planck equation $H_t-I=\int_0^t \mathfrak{A}H_sds$ and defines the Brownian motion on $U(d)$.\\
For any smooth functions $f_1$ and $f_2$ on $U(d)$, we define the energy density  $$\Gamma(f_1,f_2)=\frac{1}{2}[\mathfrak{A}[f_1f_2]-f_1\mathfrak{A}f_2-f_2\mathfrak{A}f_1].$$
It has a derivation property:
$$\Gamma(f_1f_2,f_3)=f_2 \Gamma(f_1,f_3)+f_1\Gamma(f_2,f_3).$$
Note that for any $n$-tuple of smooth functions $f_i$, and any $m$-tuple $g_j$
\begin{equation}\label{gagaz}
\mathfrak{A}[\prod f_i]=\sum_i \mathfrak{A}[f_i]  \prod_{j\neq i}f_j +2\sum_{i< k}\Gamma(f_j,f_k)\prod_{l\neq j,k}f_l,
\end {equation}
\begin {equation}\label{gagaz2}
\Gamma(\prod_i f_i,\prod_j g_j)=\sum_{i,j}\Gamma(f_i, g_j)\prod _{i'\neq i}f_{i'}\prod _{j'\neq j}g_{j'}.
\end {equation}

The Casimir operator extends naturally to an operator $\mathfrak{A}^{(\mathcal{X})}$ on smooth functions on the gauge group $\mathcal{X}$.
$\mathfrak{A}^{(\mathcal{X})}$ is the sum $\sum_{x\in X} \mathfrak{A}^{(x)} $ of the Casimir operators $\mathfrak{A}^{(x)}$ acting on each coordinate of $\mathcal{X}=U(d)^X$. We can define in the same way the Casimir operator $\sum_{\{x,y\}\in E}\mathfrak{A}^{(\{x,y\})} $ on smooth functions on $U(d)^E$ which induces after choosing an arbitrary orientation $o$ on each edge, an operator $\mathfrak{A}^{(\mathcal{C},o)}$ on smooth functions on $\mathcal C(\mathcal{G}, U(d))$. 
 Tensor products of heat convolution semigroup on $U(d)$ naturally define one on $U(d)^E $ and hence, once an orientation $o$ has been chosen, a heat semigroup $H_t^{\mathcal{C},o}$ on $\mathcal{C}(\mathcal{G}, U(d))$: Let $u_{t}(\{x,y\})$ be independent $H_t$-distributed r.v. indexed by $E$. Then, given a connection $m$, to get a sample of the distribution $H_t^{\mathcal{C},o}(m,dm')$, set, for each positively oriented edge $(x,y)$, $m^t_{x,y}= u_t(\{x,y\})m_{x,y}$ and $m^t_{y,x}= m_{x,y}^*u^*_t(\{x,y\})=m_{y,x}u_t^*(\{x,y\})$.

One can then check that the distribution of the random connection defined by $m^t$ is unchanged if we change the orientation choice: If we reverse the orientation of one positively oriented edge $(x,y)$, we now have $m^t_{x,y}=m_{x,y}u_t(\{x,y\})^*$. Replace $u_t(\{x,y\})$ by $m_{x,y} u^*_t(\{x,y\})m_{x,y}^*$, which has the same distribution, to get this new value.
Moreover, if we perform a gauge transformation in which $m$ is replaced by  $m^{(g)}$, with  ${m}^{(g)}_{x,y}=g_x^*m_{x,y}g_y$,
we can replace, for each positively oriented $(x,y)$ $u_t(\{x,y\})$ by $g_x^*u_t(\{x,y\})g_x$, which does not change the joint distribution, to get $m^{t(g)}$ instead of $m^t$. Therefore, the distribution of the modulus of this random connection depends only on the modulus of $m$.  We denote by $H_t^{(\mathcal{C})}$ the heat semigroup on connections defined in this way.  It is now clear that the operator  $\mathfrak{A}^{(\mathcal{C},o)}$ does not depend on the orientation choice. We will denote it by $\mathfrak{A}^{(\mathcal{C})}$. We have also proved that as $H_t^{(\mathcal{C})}$, it acts on functions defined on the moduli space. In what follows, we consider its restriction to such functions.

$\Gamma^{(\mathcal{X})}$ and $\Gamma^{(\mathcal{C})}$ are defined as $\Gamma$ and satisfy with $\mathfrak{A}^{(\mathcal{X})}$ and  $\mathfrak{A}^{(\mathcal{C})}$ respectively the equations (\ref{gagaz}) and (\ref{gagaz2}). 
Most importantly, the next theorem also relates the Casimir operators on connections and on the gauge group to generators of split-and-merge processes on ensembles of discrete loops defined in section \ref{SM}.
%
We shall use the notation: $$V=\sum_x  N_x^2,\;\;\;S=\sum_x N_x=\sum_{x,y}N_{x,y},$$ $$V^+=\sum_{\{x,y\}}(N_{x,y}^2+N_{y,x}^2),\;\;V^-=2\sum_{\{x,y\}}N_{x,y}N_{y,x}.$$
\begin{theorem}\label{chacha}
a) For any connection modulus $A$, and collection of geodesic loops $\tilde{\mathcal{L}}$, setting  $\tau_A(\tilde{\mathcal{L}})=\prod_{l\in \tilde{\mathcal{L}}} Tr(h_A(l))$ and $\tau_A( \emptyset)= d$, we have the intertwining property \footnote{Note that on the right hand side, the Casimir operator acts on $\tau_A(\tilde{\mathcal{L}})$ with $\tilde{\mathcal{L}}$ fixed and on the left hand side, the split and merge operators act on $\tau_A(\tilde{\mathcal{L}})$ with $A$ fixed.}
$$\mathfrak{A}^{(\mathcal{C})}\tau_A(\tilde{\mathcal{L}})=-[{(d-1)S+V^+ -V^-}+B^{SM+}-B^{SM-}]\tau_A(\tilde{\mathcal{L}}).$$
In particular, for any collection of
geodesic loops $\mathcal {L}_{flw}$ whose induced network is a flow,\\
$$\mathfrak{A}^{(\mathcal{C})}\tau_A(\mathcal{L}_{flw})=-[{(d-1)S+V^+}+B^{SM+}]\tau_A(\mathcal{L}_{flw}).$$
b) For any gauge group element $\{g\}$ and collection of loops $\mathcal{L}$, set, with the notation introduced in remark \ref{didon}, $T_{\{g\}}(\mathcal{L})=\prod_{l\in  \mathcal{L}} Tr (q_{\{g\}}(l))$. Then:
$$\mathfrak{A}^{(\mathcal{X})}T_{\{g\}}(\mathcal{L})=-[(d-1) S +V +B^{SM}]T_{\{g\}}(\mathcal{L}).$$
\end{theorem}
\begin{remark}\label{duee}
If $\tilde{\mathcal{L}}$ is a collection of non-intersecting circuits, $\mathfrak{A}^{(\mathcal{C})}\tau_A(\tilde{\mathcal{L}})=-d\,S(\tilde{\mathcal{L}}) \tau_A(\tilde{\mathcal{L}}).$
\end{remark}

\begin{remark}
a) In the Abelian case $(d=1)$, $A$ can be represented by a one-form $\omega$, with $\tau_A(\mathcal{L})=\prod_{x,y}e^{\sqrt{-1}\omega_{x,y}N_{x,y}(\mathcal{L})}. $ We can check directly that the theorem holds, but $B^{SM\pm}$ vanish.\\
 b) By remark \ref{conrk} b), and theorem \ref{coinc}, the functions $\tau_A(\tilde{\mathcal{L}})$ defined in a) determine the connection modulus $A$. \\
\end{remark}
\begin{proof}
We start the proof of the proposition with the following:
\begin{lemma}\label{sonia}
Let $a$, $b$ and $g$ be unitary matrices in $U(d)$,  and let $t_a(g)$, $r_a(g)$, $u_{a,b}(g)$, $v_{a,b}(g)$ and $w_{a,b}(g)$ denote respectively $Tr(ga) $, $Tr(g^*a)$, $Tr(gagb) $, $Tr(g^*ag^*b) $ and $Tr(gag^*b) $. Then:
$$\mathfrak{A}t_a= -d \,t_a,\; \mathfrak{A}r_a=-d \, r_a$$
$$\mathfrak{A}u_{a,b}=-2 d\, u_{a,b}- 2t_at_b, \; \mathfrak{A}v_{a,b}=-2 d \,v_{a,b}- 2r_ar_b,\; \mathfrak{A}w_{a,b}=-2d\, w_{a,b}+2 t_a r_b ,$$
$$\Gamma(t_a,t_b)=-u_{a,b} \; ,\Gamma(r_a,r_b)= -v_{a,b}\;and\; \Gamma(t_a,r_b)=Tr(ab).$$
\end{lemma}
\begin{proof}
 For $1\leq i\leq d$, let $u^{(i)}$ denote the matrix whose only non zero entry is $u_{ii}^{(i)}=\sqrt{-1}$. For $1\leq i<j \leq d$, let $u^{(i,j)}$ and $v^{(i,j)}$ denote the matrices whose only non zero entries are respectively $u^{(i,j)}_{ij}=-u^{(i,j)}_{ji}=\frac{1}{\sqrt {2}}$ and $v^{(i,j)}_{ij}=v^{(i,j)}_{ji}=\frac{\sqrt{-1}}{\sqrt {2}}$. One can check easily that these $d^2$ antihermitian matrices define an orthonormal base of the Lie algebra equipped with the scalar product defined by the Killing form, and the following relation
$$\sum_i [u^{(i)}]^2 +\sum_{i<j} [u^{(i,j)}]^2 +\sum_{i<j} [v^{(i,j)}]^2 =-dI$$ which implies the first two identities.
Then, as $u_{a,b}(g)=Tr(gagb)$ it follows that $\mathfrak{A}u_{a,b}(g)=-2\, d Tr(gagb) +2\sum_i Tr[u^{(i)}gau^{(i)}gb] + 2\sum_{i<j}Tr[u^{(i,j)}gau^{(i,j)}gb]+2\sum_{i<j}Tr[v^{(i,j)}gav^{(i,j)}gb]$ \\
$=-2 \,d Tr(gagb) -2\sum_i [ga]_{ii}[gb]_{ii} + \sum_{i<j}( [ga]_{ij}[gb]_{ij}+ [ga]_{ji}[gb]_{ji} -[ga]_{ii}[gb]_{jj}- [ga]_{jj}[gb]_{ii})  - \sum_{i<j}( [ga]_{ij}[gb]_{ij}+ [ga]_{ji}[gb]_{ji} +[ga]_{ii}[gb]_{jj}+ [ga]_{jj}[gb]_{ii}) $ \\
$=-2 \,d Tr(gagb)-2Tr(ga)Tr(gb).$\\
The fourth and fifth identities are proved in the same way:\\
$\mathfrak{A}v_{a,b}(g)=-2\, d Tr(g^*ag^*b) +\sum_i Tr[g^*u^{(i)}ag^*u^{(i)}b] + 2\sum_{i<j}Tr[g^*u^{(i,j)}ag^*u^{(i,j)}b]+2\sum_{i<j}Tr[g^*v^{(i,j)}ag^*v^{(i,j)}b]$ \\
$=-2 \,d Tr(g^*ag^*b) -2\sum_i [ag^*]_{ii}[bg^*]_{ii} + \sum_{i<j}( [ag^*]_{ij}[bg^*]_{ij}+ [ag^*]_{ji}[bg^*]_{ji} -[ag^*]_{ii}[bg^*]_{jj}- [ag^*]_{jj}[bg^*]_{ii})  - \sum_{i<j}
( [ag^*]_{ij}[bg^*]_{ij}+ [ag^*]_{ji}[bg^*]_{ji} +[ag^*]_{ii}[bg^*]_{jj}+ [ag^*]_{jj}[bg^*]_{ii}) $ \\
$=-2 \,d Tr(g^*ag^*b)-2Tr(ag^*)Tr(bg^*).$\\
$\mathfrak{A}w_{a,b}(g)=-2 \,d Tr(gag^*b) -2\sum_i Tr[u^{(i)}gag^*u^{(i)}b] - 2\sum_{i<j}Tr[u^{(i,j)}gag^*u^{(i,j)}b]-2\sum_{i<j}Tr[v^{(i,j)}gag^*v^{(i,j)}b]$ \\
$=-2\, d Tr(gagb) -2\sum_i [gag*]_{ii}[b]_{ii} + \sum_{i<j}( [gag*]_{ij}[b]_{ij}+ [gag*]_{ji}[b]_{ji} -[gag^*]_{ii}[b]_{jj}- [gag^*]_{jj}[b]_{ii})  - \sum_{i<j}
( [gag^*]_{ij}[b]_{ij}+ [gag^*]_{ji}[b]_{ji} +[gag^*]_{ii}[b]_{jj}+ [gag^*]_{jj}[b]_{ii}) $ \\
$=-2 \,d Tr(gag^*b)-2Tr(gag^*)Tr(gb)=2 \,d Tr(gag^*b)-2Tr(a)Tr(b).$\\
Note that $\mathfrak{A}w_{a,b}$ vanishes if $a$ or $b=I$. \\Let us now prove the sixth identity:
$ \Gamma(t_a,t_b)=\sum_i Tr[u^{(i)}ga]Tr[u^{(i)}gb] + \sum_{i<j}Tr[u^{(i,j)}ga]Tr[u^{(i,j)}gb] + \sum_{i<j}Tr[v^{(i,j)}ga][v^{(i,j)}gb] $\\
$=-\sum_i [ga]_{ii}[gb]_{ii} + \frac{1}{2}\sum_{i<j}( [ga]_{ij}-[ga]_{ji})( [gb]_{ij}-[gb]_{ji})-\frac{1}{2}\sum_{i<j}( [ga]_{ij}+[ga]_{ji})( [gb]_{ij}+[gb]_{ji})  $ \\
$=-\sum_i [ga]_{ii}[gb]_{ii} -\sum_{i<j}([ga]_{ij}[gb]_{ji}+[gb]_{ij}[ga]_{ji})=-Tr(gagb)$. \\
The seventh identity is proved in the same way. Let us finally prove the last one:\\
$ \Gamma(t_a,r_b)=-\sum_i Tr[u^{(i)}ga]Tr[g^*u^{(i)}b] - \sum_{i<j}Tr[u^{(i,j)}ga]Tr[g*u^{(i,j)}b] - \sum_{i<j}Tr[v^{(i,j)}ga][g^*v^{(i,j)}b] $\\
$=\sum_i [ga]_{ii}[bg^*]_{ii} - \frac{1}{2}\sum_{i<j}( [ga]_{ij}-[ga]_{ji})( [bg^*]_{ij}-[bg^*]_{ji})+\frac{1}{2}\sum_{i<j}( [ga]_{ij}+[ga]_{ji})( [bg^*]_{ij}+[bg^*]_{ji})  $ \\
$=-\sum_i [ga]_{ii}[bg^*]_{ii} +\sum_{i<j}([ga]_{ij}[bg^*]_{ji}+[ga]_{ji}[bg^*]_{ij})=Tr(gabg^*)=Tr(ab).$
\end{proof}

Let us first complete the proof of theorem \ref{chacha} b), which is easier, using definitions and notation given in proposition \ref{SM}, remark \ref{didon}, and in the introduction to the present proposition. It follows by linear combination from the identity:
\begin {equation}\label{elis}
-\mathfrak{A}^{(x)}T_{\{g\}}(\mathcal{L})=(d-1) \,N_x(\mathcal{L})T_{\{g\}}(\mathcal{L})+N_x^2(\mathcal{L})[Q_{(x)}T_{\{g\}}](\mathcal{L})
\end {equation}
 in which $N_x(\mathcal{L})$ denotes $\sum_{l\in \mathcal{L}} N_x(l)$.\\
To prove (\ref{elis}), note first that, using (\ref{gagaz}), $\mathfrak{A}^{(x)}T_{\{g\}}(\mathcal{L})$ can be decomposed into a sum of terms corresponding to the action of $\mathfrak{A}^{(x)}$ on the holonomy trace of a single loop $\xi$, (with $N_x(\xi)$ insertions of $\{g\}(x)$ in the expression of  $Tr(q_{\{g\}}(\xi))$) and the actions of $\Gamma^{(x)}$ on the holonomy traces of a pair of loops $(\xi_1, \xi_2)$ (with $N_x(\xi_1)N_x(\xi_2)$ pairs of  $\{g\}(x)$ insertions if $\xi_1 \neq \xi_2$ and $\frac{1}{2}N_x(\xi)(N_x(\xi)-1)$ unordered pairs if $\xi_1 =\xi_2)=\xi$. As $\mathfrak{A}^{(x)}$ acts only on single and pairs of insertions, the first and third identities in lemma \ref{sonia} above are easily extended to show that for any discrete loop $\xi=(\xi_i)$, $-\mathfrak{A}^{(x)}Tr(q_{\{g\}}((\xi))$ equals $$d\,N_x(\xi)Tr(q_{\{g\}}(\xi) +2 \sum_{i_1 < i_2,\, \xi_{i_1}=\xi_{i_2}=x} Tr(q_{\{g\}}(\xi_{[i_1,i_2[})Tr(q_{\{g\}}(\xi_{[i_2,i_1 +p(\xi)[}), $$ in which $\xi_{[i_1,i_2[}$ and $\xi_{[i_2,i_1+p(\xi)[}$ denote the loops obtained by splitting $\xi$ at points of indices $i_1$ and $i_2$ and connecting the endpoints of each part.\\
Moreover, the sixth identity gives that for any pair of discrete loop $(\xi^{(1)},\xi^{(2)})$, $\Gamma^{(x)}(Tr(q_{\{g\}}(\xi^{(1)})),Tr(q_{\{g\}}(\xi^{(2)})))=\sum_{(i_1 ,i_2),\, \xi^{(1)}_{i_1}=\xi^{(2)}_{i_2}=x}Tr(q_{\{g\}}(\xi^{(1,i_1)}\cdot \,\xi^{(2,i_2)}))$ in which $\xi^{(1,i_1)}\cdot \,\xi^{(2,i_2)}$ denotes the concatenation of the loops $\xi^{(1)}$ and $\xi^{(2)}$ at the base points of respective indices $i_1$ and $i_2$.
Using equation (\ref{gagaz}) and the definition of $Q_{(x)}$ in section \ref{SM} (which includes diagonal entries $Q_{(x)}(\mathcal{L},\mathcal{L})=\frac{1}{N_x(\mathcal{L})}$, giving $d-1$ instead of $d$ in front of the first term), this completes the proof of equation (\ref{elis}) and therefore the proof of theorem \ref{chacha} b).
 
Let us finally complete the proof of theorem \ref{chacha} a), using the notations of section 2. After noting that $h_A(l)=h_A(l^R)$ implies that $K^-_{\{x,y\}}\tau_A=K^{-\,R}_{\{x,y\}}\tau_A$, it follows by linear combination from the identity:
\begin {equation}\label{elisa}
-\mathfrak{A}^{(\{x,y\})}\tau_A(\tilde{\mathcal{L}})=[(d-1) \,N_{\{x,y\}}+N_{x,y}^2K^+_{(x,y)}+N_{y,x}^2K^+_{(y,x)}-2N_{x,y}N_{y,x}K^-_{\{x,y\}}]\tau_A(\tilde{\mathcal{L}}).
\end {equation}
The proof of (\ref{elisa}) is almost the same as the proof of identity (\ref{elis}), but, after choosing any representative $m$ of the connection modulus $A$ and an orientation $o$, one has to consider the insertions of $m(x,y)$ and $m^*(x,y)$ in $\tau_A(\mathcal{L})$, now using all identities in lemma \ref{sonia}, applied to three different types of pairs of crossing. The only essential difference lies in the occurrence of the term $-K^-_{\{x,y\}}$ which is arising from the expressions for $\mathfrak{A}w_{a,b}(g)$ and $\Gamma(t_a,r_b)(g)$ given in lemma  \ref{sonia}, in which $g$ and $g^*$ cancel each other.\\
For any geodesic loop $\gamma$, if $(x,y)$ is positively oriented,\\$-\mathfrak{A}^{(\{x,y\})}\tau_A(\gamma)
=dN_{x,y}(\gamma)\tau_A(\gamma)+ dN_{y,x}(\gamma)\tau_A(\gamma)$\\
$+2\sum_{\{i_1<i_2, \gamma_{i_1}=\gamma_{i_2}=x, \gamma_{i_1+1}=\gamma_{i_2}+1=y\}}\tau_A( \gamma_{[i_1,i_2[})\tau_A( \gamma_{[i_2,i_1+p(\gamma)[})$\\
$+2\sum_{\{i_1<i_2, \gamma_{i_1}=\gamma_{i_2}=y, \gamma_{i_1+1}=\gamma_{i_2}+1=x\}}\tau_A( \gamma_{[i_1,i_2[})\tau_A( \gamma_{[i_2,i_1+p(\gamma)[})$\\
$-2\sum_{\{i_1< i_2,  \gamma_{i_1} = \gamma_{i_2+1}=x, \gamma_{i_1+1}= \gamma_{i_2}=y\}}\tau_A( \gamma_{[i_1+1,i_2[})\tau_A( \gamma_{[i_2+1,i_1+p(\gamma)[})$\\
$-2\sum_{\{i_1< i_2, \gamma_{i_1} = \gamma_{i_2+1}=y, \gamma_{i_1+1}= \gamma_{i_2})=x\}}\tau_A( \gamma_{[i_1+1,i_2[})\tau_A( \gamma_{[i_2+1,i_1+p(\gamma)[})$.\\
We get the first four terms from the four first identities in lemma \ref{sonia}, in the same order, and the last two terms, in which two opposite edges are cancelled, from the fifth one. Note that as expected, we get the same result if $(x,y)$ is negatively oriented, but we get the first term from the second identity, the second from the first, the third from the fourth and the fourth from the third.\\
 This expression can be rewritten as $dN_{x,y}(\gamma)\tau_A(\gamma)+ dN_{y,x}(\gamma)\tau_A(\gamma)+$
$${2\sum}_{\{\gamma^{(1)},\gamma^{(2)}\}\in Split_{\{x,y\}}^+(\gamma)}\tau_A(\{\gamma^{(1)},\gamma^{(2)}\})
-{2\sum}_{\{\gamma^{(1)},\gamma^{(2)}\}\in Split_{\{x,y\}}^-(\gamma)}\tau_A(\{\gamma^{(1)},\gamma^{(2)}\})$$ in which $Split_{\{x,y\}}^{\pm}(\gamma)$ are the two collections of splitting outputs (unordered pairs of geodesic loops) defined just above ($Split_{\{x,y\}}^-$ being the one involving edge cancellation). We verify that this expression depends only on the connection modulus $A$.

From the definitions of $K_{(x,y)}^+$ and $K_{\{x,y\}}^-$ given in section \ref{SM}, (noting that $K_{\{x,y\}}^+$ also includes diagonal entries giving $d-1$ instead of $d$ in front of the first term), this completes the proof of equation (\ref{elisa}) for a single geodesic loop $\gamma$. \\
Note that $\sum_{\{x,y\}\in E}N_{\{x,y\}}(\gamma)=p(\gamma)$. Then denoting the collections \\$\oplus Split_{\{x,y\}}^{\pm}(\gamma)$ \footnote{We use $\oplus$, not $\cup$, as we deal with collections and multiplicities are added} by $Split^{\pm}(\gamma)$, the term $-\mathfrak{A}^{(\mathcal{C})}\tau_A(\gamma)$ can be rewritten:
\begin {equation}\label{elisa}
d\,p(\gamma)\tau_A(\gamma)+\sum_{\{\gamma^{(1)},\gamma^{(2)}\}\in Split^+(\gamma)}2\tau_A(\{\gamma^{(1)},\gamma^{(2)}\})
-\sum_{\{\gamma^{(1)},\gamma^{(2)}\}\in Split^-(\gamma)}2\tau_A(\{\gamma^{(1)},\gamma^{(2)}\}).
\end{equation}
For any pair of geodesic loops $(\gamma^{(1)},\gamma^{(2)})$, we get from the three last identities in lemma \ref{sonia} that\\ $$\Gamma^{(\mathcal{C})}(\tau_A(\gamma^{(1)}),\tau_A(\gamma^{(2)}))
=-\sum_{( \gamma^{(1)}_{i_1},\gamma^{(1)}_{i_1+1}) =( \gamma^{(2)}_{i_2},\gamma^{(2)}_{i_2+1})}\tau_A(\gamma^{(1,i_1)}\cdot \,\gamma^{(2,i_2)})$$
 $$+\sum_{( \gamma^{(1)}_{i_1},\gamma^{(1)}_{i_1+1}) =( \gamma^{(2)}_{i_2+1},\gamma^{(2)}_{i_2})}\tau_A(\gamma^{(1,i_1)}\cdot \,\gamma^{(2,i_2+1)})$$
 in which $\gamma^{(1,i)}\cdot \,\gamma^{(2,j)}$ denotes the concatenation, with cancelation of inverse edges, of the geodesic loops $\gamma^{(1)}$ and $\gamma^{(2)}$ at the base points of respective indices $i$ and $j$.\\
Note that in the last term, two opposite edges are cancelled. Again, \\$\Gamma^{(\mathcal{C})}(\tau_A(\gamma^{(1)}),\tau_A(\gamma^{(2)}))$ can be rewritten as
\begin{equation} \label{sand}
-\sum_{\gamma \in Merge^+(\gamma^{(1)},\gamma^{(2)})}\tau_A(\gamma)
+\sum_{\gamma \in Merge^-(\gamma^{(1)},\gamma^{(2)})}\tau_A(\gamma)
\end{equation}
 in which $Merge^{\pm}(\gamma^{(1)},\gamma^{(2)})$ are the two collections of merging outputs defined just above.

Considering a finally collection of geodesic loops instead of a single one, the extension of equation (\ref{gagaz}) to $\mathfrak{A}^{(\mathcal{C})}$ allows to complete the proof of the first identity in a). 
\end{proof}

 Denote by $\langle\cdot\rangle$ the integration against the Haar measure on $U(d)$. For any pair of smooth functions $f$ and $g$ on $U(d)$, and any antihermitian matrix $E$, we have 
$\langle f \mathfrak{L}_E g \rangle=-\langle g \mathfrak{L}_E f \rangle$, and consequently, $\langle f \mathfrak{A} g\rangle=\langle g \mathfrak{A} f \rangle$. In particular, $\langle \mathfrak{A} f \rangle$ vanishes. This self-adjointness property extends to the pairs of  Haar measures and Casimir operators ($\langle\cdot\rangle_{(X)}$, $\mathfrak{A}^{(X)}$) and ($\langle\cdot\rangle_{(\mathcal{C})}$, $\mathfrak{A}^{(\mathcal{C})}$) defined respectively on the gauge group $U(d)^X$ and on $U(d)$-connections. 
Then, from theorem \ref{chacha} we get the following identities
\begin{corollary}\label{swingd}
For any set of geodesic loops $\tilde{\mathcal{L}}$,
 $$\langle [{(d-1)S+V^+-V^-}+B^{SM+}-B^{SM-}]\tau_A(\tilde{\mathcal{L}}) \rangle_{(\mathcal{C})}=0. $$

For any set of loops $\mathcal{L}$,
 $$\langle [(d-1)S +V +B^{SM}]T_{\{g\}}(\mathcal{L}) \rangle_{(X)}=0.$$
\end{corollary}
\begin{remark}\label{bibif}
Using the extension of equation (\ref{gagaz}) and (\ref{gagaz2}) to $\mathfrak{A}^{(\mathcal{C})}$, formulas (\ref{elisa}) and (\ref{sand}) extend to to products of traces of geodesic loop holonomies. 
Using the notation introduced in the proof of theorem \ref{chacha}, we can reformulate the first identity of this corollary as follows: For any finite sequence of geodesic loops $(\gamma_i)$,
$$(d-1)\langle \sum_i p(\gamma_i)\prod_i \tau_A(\gamma_i) \rangle_{(\mathcal{C})}=\langle \;2\sum_{i}[\sum_{\{\gamma^{(1)},\gamma^{(2)}\}\in Split^-(\gamma_i)}\tau_A(\gamma^{(1)})\tau_A(\gamma^{(2)})$$
$$-\sum_{\{\gamma^{(1)},\gamma^{(2)}\}\in Split^+(\gamma_i)}\tau_A(\gamma^{(1)})\tau_A(\gamma^{(2)})]\prod_{i'\neq i}\tau_A(\gamma_i')+\sum_{i,j}\langle [ \sum_{\gamma \in Merge^-(\gamma_i,\gamma_j)}\tau_A(\gamma)$$
$$-\sum_{\gamma \in Merge^+(\gamma_i,\gamma_j)}\tau_A(\gamma)]\prod_{i'\neq i,j}\tau_A(\gamma_i')
\;\rangle_{(\mathcal{C})}. $$
\end{remark}

\begin{remark}\label{mopka}
 For any $U(d)$-connection modulus $A$, denoting by $G_0$ the set of geodesic loops,
we can apply these result to the Yang-Mills weights defined in section \ref{ymil}. Set $P(A)= \sum_{\eta\in \mathcal{P} }\tau_A(\eta)$. Then, since plaquettes, being simple circuits, cannot split, we have $-\mathfrak{A}^{(\mathcal{C})}P(A)=dP(A)$. Note also that two adjacent plaquettes $\eta_1\,, \eta_2$ can be merged positively  or negatively (i.e. with cancelation of the common edge) depending on whether their common edge is crossed in the same direction (we write $\eta_1\sim_+\eta_2$) or not (we write $\eta_1\sim_-\eta_2$). If we denote the output of this merge by $\eta_1\eta_2$, then we have\\ $\Gamma^{(\mathcal{C})}(P(A),P(A))=-\sum_{\eta_1\sim_+\eta_2}\tau_A(\eta_1\eta_2)+\sum_{\eta_1\sim_-\eta_2}\tau_A(\eta_1\eta_2)$.
Consequently, defining as in section \ref{ymil} the Yang Mills weights $Y_k(A)$ by $e^{\frac{k}{d}[P(A)-dn(n-1)L^n] }$ (in which $n(n-1)L^n$ is the number of plaquettes), we have $$\mathfrak{A}^{(\mathcal{C})}Y_k(A)=[-\frac{k}{d} (d-1)P(A)+\frac{k^2}{d^2}(-\sum_{\eta_1\sim_+\eta_2}\tau_A(\eta_1\eta_2)+\sum_{\eta_1\sim_-\eta_2}\tau_A(\eta_1\eta_2))]Y_k(A).$$
\end{remark}
We will now see how the intertwining relation of theorem \ref{chacha} between generators on connections and generators on collections of loop can be extended to heat kernels.
Note first that $H_t^{(\mathcal{C})}$ inherit from $H_t$ the verification of Fokker-Planck equations with respect to the corresponding generator $\mathfrak{A}^{(\mathcal{C})}$. A semigroup satisfying these equations is necessarily unique: Indeed, if $P'_t$ verifies the backward equation and $P''_t$ the forward equation, $\frac{d}{ds}P'_sP''_{t-s}$ vanishes.

\begin{corollary}\label{pete}
Let $H_t^{SM}$ be the semigroup defined by the generator $B^{SM+}+B^{SM-}$ and, for any finite collection of geodesic loops $\tilde{\mathcal{L}}_0$, denote by 
$\mathbb{P}_{\tilde{\mathcal{L}}_0}^{SM}$ the distributions of the corresponding split-and-merge Markov chain $(\tilde{\mathcal{L}}_s,\,s>0)$ on collections of geodesic loops starting from $\tilde{\mathcal{L}}_0$. Then, for any connection modulus $A_0$, we have:\\
 $$\int H_t^{(\mathcal{C})}(A_0,dA)\tau_{A}(\tilde{\mathcal{L}}_0)= \mathbb{E}_{\tilde{\mathcal{L}}_0}^{SM}[(-1)^{\mathfrak{m}^+_t}e^{\int_0^t [-(d-1)S+2V^-](\tilde{\mathcal{L}}_s)ds}\tau_{A_0}(\tilde{\mathcal{L}}_t)]$$
 in which $\mathfrak{m}^+_t$ denotes the number of positive merges or splits between $0$ and $t$.\\ 
 %
In particular, denoting by $\mathbb{P}_{\tilde{\mathcal{L}}_0}^{SM+}$ the distributions of the split-and-merge Markov chain generated by $B^{SM+}$  starting from $\tilde{\mathcal{L}}_0$,  for any collection of geodesic loops $\mathcal{L}^{flw}$ whose induced network is a flow, we have
  $$\int H_t^{(\mathcal{C})}(A_0,dA)\tau_A(\mathcal{L}^{flw})= e^{-t(d-1)S(\mathcal{L}^{flw})}\mathbb{E}_{\tilde{\mathcal{L}}^{flw}}^{SM+}[(-1)^{\mathfrak{m}^+_t}\tau_{A_0}(\tilde{\mathcal{L}}_t)].$$  \end{corollary}
 \begin{proof}
 a) The Markov property of $(\tilde{\mathcal{L}}_s\,s>0)$ implies that the matrices $\tilde{H}_t^{SM}(\tilde{\mathcal{L}}',\tilde{\mathcal{L}}'')$ indexed by finite collections of geodesic loops defined by: $$\tilde{H}_t^{SM}(\tilde{\mathcal{L}}',\tilde{\mathcal{L}}'')=\mathbb{E}_{\tilde{\mathcal{L}}'}^{SM}[(-1)^{\mathfrak{m}_t}e^{-\int_0^t [-(d-1)S+2V^-](\tilde{\mathcal{L}}_s)ds}1_{\tilde{\mathcal{L}}'_t=\tilde{\mathcal{L}}''}]$$ form a semigroup. Noting that  the first jump $T_1$ of the process of $(\tilde{\mathcal{L}}_s,\,s>0)$ starting from $\tilde{\mathcal{L}}'$ occurs at a random exponential time of mean $( V^+ +V^-)(\tilde{\mathcal{L}}')$, and that the probability of a second jump before time $\epsilon$ can be bounded by $O(\epsilon^2)$, we have:\\
 $\tilde{H}_{\epsilon}^{SM}(\tilde{\mathcal{L}}',\tilde{\mathcal{L}}'') =\mathbb{E}^{SM}_{\tilde{\mathcal{L}}'}(e^{\int_0^{\epsilon} [-(d-1)S+V^- -V^+](\tilde{\mathcal{L}}_s)ds}[1_{T_1>\epsilon}1_{\tilde{\mathcal{L}}'=\tilde{\mathcal{L}}''}+1_{T_1\leq\epsilon}$\\
  $\times \sum_{x,y}[-\frac{N_{x,y}^2(\tilde{\mathcal{L}}')}{(V^++V^-)(\tilde{\mathcal{L}}')}K^{+}_{(x,y)}(\tilde{\mathcal{L}}',\tilde{\mathcal{L}}'')+\frac{N_{x,y}N_{y,x}(\tilde{\mathcal{L}}')}{(V^++V^-)(\tilde{\mathcal{L}}')}K^{-,R}_{\{x,y\}}(\tilde{\mathcal{L}}',\tilde{\mathcal{L}}'')])+O(\epsilon^2)$\\
  
  Hence, as $(d-1)S+V^+ -V^-=[V^++V^- ]+[d-1)S-2V^-]$,\\  $\tilde{H}_{\epsilon}^{SM}(\tilde{\mathcal{L}}',\tilde{\mathcal{L}}'') =1_{\tilde{\mathcal{L}}'=\tilde{\mathcal{L}}''}e^{-\epsilon [(d-1)S+V^+-V^-](\tilde{\mathcal{L}}')}$\\
 $+ \int_0^{\epsilon}e^{-s[(d-1)S+V^+-V^-](\tilde{\mathcal{L}}')-(\epsilon-s)[(d-1)S+V^+-V^-]({\mathcal{L}}'')}(V^++V^-)(\tilde{\mathcal{L}}')ds$\\
 $\times \sum_{x,y}[-\frac{N_{x,y}^2(\tilde{\mathcal{L}}')}{(V^++V^-)(\tilde{\mathcal{L}}')}K^{+}_{(x,y)}(\tilde{\mathcal{L}}',\tilde{\mathcal{L}}'')+\frac{N_{x,y}N_{y,x}(\tilde{\mathcal{L}}')}{(V^++V^-)(\tilde{\mathcal{L}}')}K^{-,R}_{\{x,y\}}(\tilde{\mathcal{L}}',\tilde{\mathcal{L}}'')]+O(\epsilon^2)$\\
 
 $=1_{\tilde{\mathcal{L}}'=\tilde{\mathcal{L}}''}(1-\epsilon (d-1)S(\tilde{\mathcal{L}'}) - \epsilon (V^+ - V^-)(\mathcal{L}'))+\epsilon \sum_{x,y}[-N_{x,y}^2(\tilde{\mathcal{L}}')K^{+}_{(x,y)}(\tilde{\mathcal{L}}',\tilde{\mathcal{L}}'')+N_{x,y}N_{y,x}(\tilde{\mathcal{L}}')K^{-\,R}_{\{x,y\}}(\tilde{\mathcal{L}}',\tilde{\mathcal{L}}'')]+O(\epsilon^2)$\\
 
 $=1_{\tilde{\mathcal{L}}'=\tilde{\mathcal{L}}''}+ \epsilon [(-(d-1)S-V^++V^-)(\tilde{\mathcal{L}}')- B^{SM+}(\tilde{\mathcal{L}}',\tilde{\mathcal{L}}'')+B^{SM-}(\tilde{\mathcal{L}}',\tilde{\mathcal{L}}'') ]+O(\epsilon^2)
  .$
 Hence this semigroup solves the Fokker-Planck equation $$\tilde{H}_t^{SM}- 1=\int_0^t \tilde{H}_s^{SM}[{-(d-1)S-V^++V^-}-B^{SM+}+B^{SM-}]ds=\int_0^t \tilde{H}_s^{X,MS}\mathfrak{A}^{(\mathcal{C})}ds.$$ Therefore it coincides with $H_t^{(\mathcal{C})}$ (as product of traces determine collections of loops), which we just showed to be the unique solution of this equation. 
 Finally, note as before that negative split or merge cannot occur on collections of loops inducing a flow, so that $S$ and $V^+$ are constant, $V^-=0$ and $B^{SM-}=0$. 
  \end{proof}
\begin {remark}
The process $\tilde{\mathcal{L}}_s$ is always absorbed in finite time by the recurrence class defined by the flow induced by $\tilde{\mathcal{L}}_0$.
\end {remark}
\begin {remark}
A similar result can be derived from theorem \ref{chacha} b) for the heat semigroup on the gauge group.
\end {remark}

\section{Deformation and Marchenko-Migdal equation}

Consider a second order differential operator $\mathfrak{T}$ defined on an algebra $\mathcal{D}$ of smooth functions, without zero-order term and self-adjoint with respect to some measure $m$. For any non-vanishing element  $h$ of $\mathcal{D}$, $\mathfrak{T}$ can be modified by conjugacy in order to obtain an operator which is self-adjoint relatively to an equivalent measure with density $h^2$. We define $\Gamma^{}$ and $\mathfrak{T}_{(h)}$ such that for any function $f$ in $\mathcal{D}$: $\Gamma(h,f)= \frac{1}{2}[\mathfrak{T}[hf]-f\mathfrak{T}h-h\mathfrak{T}f]$ and 
$\mathfrak{T}_{(h)}f=\frac{1}{h}(\mathfrak{T}[hf]-f\mathfrak{T}h)=\mathfrak{T}f+\frac{2}{h}\Gamma(h,f)=\mathfrak{T}f+2\Gamma(\log(h),f)$.\\
$\mathfrak{T}$ and $\mathfrak{T}_{(h)}$ satisfy formula (\ref{gagaz}). Moreover,
for any pair $f,g\in \mathcal{D}$, $-\int f \mathfrak{T}gdm=\int\Gamma(f,g)dm$ and therefore, $-\int f \mathfrak{T}_{(h)}gh^2dm=\int\Gamma(fh,gh)dm-\int\Gamma(h,fgh)dm
=\int\Gamma(f,g)h^2dm.$\\

 \begin{remark}\label{fkpl}
 If $\mathfrak{T}$ generates a semigroup $U_t$ which operates on $\mathcal{D}$, satisfies Fokker-Planck equations, and is associated with a diffusion process $x_t$, the semigroup  $U^{(h)}_t(x,y)= \frac{h(y)}{h(x)}\mathbb{E}(e^{-\int_0^t\frac{\mathfrak{T}h}{h}(x_s)ds}|x_0=x,\,x_t=y)$ satisfies the Fokker-Planck equations relative to $\mathfrak{T}_{(h)}.$ 
  This applies to the various Casimir operators defined in the previous section.
 \end{remark}
 Consider now the case of the discrete torus $[\mathbb{Z}/L\mathbb{Z}]^{n}$ as in section \ref{ymil} and remark \ref{mopka}. We can use the Casimir operator $\mathfrak{A}^{(\mathcal{C})}$ to define an operator $\mathfrak{A}^{(\mathcal{C},k)}$ self adjoint with respect to the Yang Mills measure given (up to a multiplicative constant) by the weights $e^{\frac{k}{d}P}$.
For any smooth function $f$ on $\mathcal{C}$, set 
\begin{equation}\label{migder}
\mathfrak{A}^{(\mathcal{C},k)}f=\mathfrak{A}^{(\mathcal{C})}_{(e^{\frac{k}{2d}P})}f=\mathfrak{A}^{(\mathcal{C})}f+\frac{k}{d}\Gamma^{(\mathcal{C})}(P,f).
\end{equation}

By formula (\ref{sand}), we get that 
\begin{equation}\label{brrr}
\Gamma^{(\mathcal{C})}(P(A),\tau_A(\tilde{\mathcal{L}}))=\sum_{\eta \in \mathcal{P} }[ \sum_{\gamma \in Merge^-{(\{ \tilde{\mathcal{L}},\eta}\})}\tau_A(\gamma)-\sum_{\gamma \in Merge^+(\{\tilde{\mathcal{L}},\eta\})}\tau_A(\gamma)
].
\end{equation}

This suggests to introduce, on collections of geodesic loops, the generators $B^{D+}$ and $B^{D-}$ respectively defined as follows.
For each oriented edge $(x,y)$, the transition probability  $D^+_{(x,y)}$ is constructed by choosing one crossing uniformly among the $(x,y)$-crossings of all loops and merging the corresponding loop at this position with one plaquette chosen uniformly among the $2(n-1)$ plaquettes containing the oriented edge $(x,y)$.  $D^-_{(x,y)}$ is constructed by choosing one crossing uniformly among the $(x,y)$-crossings of all loops and merging, with cancellation and reduction, the corresponding loop at this position with one plaquette chosen uniformly among the $2(n-1)$ plaquettes containing the opposite oriented edge $(y,x)$.  They can be combined with jump rates $2(n-1)N_{x,y}$ to get two continuous time Markov chains with generator $B^{D\pm}=\sum_{(x,y)}2(n-1)N_{x,y}[D^{\pm}_{(x,y)}-I]$. 
Then we have, from equation (\ref{brrr})
$$\Gamma^{(\mathcal{C})}(P(A),\tau_A(\tilde{\mathcal{L}}))=[B^{D-}-B^{D+}]\tau_A(\tilde{\mathcal{L}}).$$ 
Consequently, from theorem \ref{chacha} and equation (\ref{migder}), we obtain the following identity:
\begin{theorem}
For any geodesic loops collection $\tilde{\mathcal{L}}$,\\
$-\mathfrak{A}^{(\mathcal{C},k)}\tau_A(\tilde{\mathcal{L}}) =[(d-1)S+V^+ -V^- +B^{SM+}-B^{SM-}+\frac{k}{d}(B^{D+}-B^{D-})]\tau_A(\tilde{\mathcal{L}}).$
\end{theorem}
\begin{remark}\label{duee2}
As in remark \ref{duee}, if the $(\gamma_i)$'s are non-intersecting circuits, the theorem's identity simplifies to
$$-\mathfrak{A}^{(\mathcal{C},k)}\tau_A(\tilde{\mathcal{L}}) =[(d-1)S+\frac{k}{d}(B^{D+}-B^{D-})]\tau_A(\tilde{\mathcal{L}}).
$$
\end{remark}

As in corollary \ref{swingd}, denoting by $ \langle\cdot\rangle_{(\mathcal{C},k)}$ the integration with respect to the Yang-Mills measure, we have 
\begin{corollary}
 $$\langle [(d-1)S+V^+ -V^-+B^{SM+}-B^{SM-}+\frac{k}{d}[B^{D+}-B^{D-}]]\tau_A(\tilde{\mathcal{L}}) \rangle_{(\mathcal{C},k)}=0. $$

\end{corollary}
\begin{remark}

As in remark \ref{bibif}, we can reformulate this identity as follows: For any finite sequence of geodesic loops  $(\gamma_i)$,
$$(d-1)\langle \sum_i p(\gamma_i)\prod_i \tau_A(\gamma_i) \rangle_{(\mathcal{C},k)}=\sum_{i}\langle \;[\;{2\sum}_{\{\gamma^{(1)},\gamma^{(2)}\}\in Split^-(\gamma_i)}\tau_A(\gamma^{(1)})\tau_A(\gamma^{(2)})$$
$$-{2\sum}_{\{\gamma^{(1)},\gamma^{(2)}\}\in Split^+(\gamma_i)}\tau_A(\gamma^{(1)})\tau_A(\gamma^{(2)})+\frac{k}{d}\sum_{\eta \in \mathcal{P} }[\sum_{\gamma \in Merge^-{(\gamma_i,\eta)}}\tau_A(\gamma)$$
$$- \sum_{\gamma \in Merge^+(\gamma_i,\eta)}\tau_A(\gamma)]\prod_{i'\neq i}\tau_A(\gamma_i'))]+\sum_{i,j}[ \sum_{\gamma \in Merge^-(\gamma_i,\gamma_j)}\tau_A(\gamma)$$
$$-\sum_{\gamma \in Merge^+(\gamma_i,\gamma_j)}\tau_A(\gamma)]\prod_{i'\neq i,j}\tau_A(\gamma_i')
\;\rangle_{(\mathcal{C},k)}. $$
As in remark \ref{duee2}, if the $(\gamma_i)$'s are non-intersecting circuits, the right hand side simplifies to $$\frac{k}{d}\sum_{\eta \in \mathcal{P} }\langle \;[\sum_{\gamma \in Merge^-{(\gamma_i,\eta)}}\tau_A(\gamma))
- \sum_{\gamma \in Merge^+(\gamma_i,\eta)}\tau_A(\gamma)]\prod_{i'\neq i}\tau_A(\gamma_i')\;\rangle_{(\mathcal{C},k)}. $$

This equation is almost identical to the Schwinger-Dyson equation previously obtained (with a similar proof, though it deals with $SO(d)$ and does not mention Casimir operators) in \cite{chater}, as an essential step in the proof of the t`Hooft expansion for large $d=n$. In continuous spaces, such equations, which originate from physics, are often referred to as Marchenko-Migdal equations (Cf. \cite{Lev2} for a proof in dimension two).
\end{remark}

We can also derive as before a heat kernel identity:
\begin{corollary}
Let $J_t^{(\mathcal{C},k)}$ be the semigroup generated by $\mathfrak{A}^{(\mathcal{C},k)}$ and $H_t^{SMD}$ the semigroup defined by the generator $B^{SM+}+B^{SM-}+\frac{k}{d}[B^{D+}+B^{D-}]$. For any finite collection of geodesic loops $\tilde{\mathcal{L}}_0$, denote by 
$\mathbb{P}_{\tilde{\mathcal{L}}_0}^{SMD}$ the distributions of the merge-and-split Markov chain $(\tilde{\mathcal{L}}_s,\,s>0)$ on collections of geodesic loops starting from $\tilde{\mathcal{L}}_0$ associated with $H_t^{SMD}$. Then:\\
 $$\int J_t^{(\mathcal{C},k)}(A_0,dA)\tau_A(\tilde{\mathcal{L}}_0)= \mathbb{E}_{\tilde{\mathcal{L}}_0}^{SMD}[(-1)^{\bar{\mathfrak{m}}^+_t}e^{\int_0^t (2\frac{k}{d}(n-1)-d+1)S(\tilde{\mathcal{L}}_s)+2V^-(\tilde{\mathcal{L}}_s)ds}\tau_{A_0}(\tilde{\mathcal{L}}_t)]$$
 in which $\bar{\mathfrak{m}}^+_t$ denotes the total number of positive merges or splits between $0$ and $t$ (now including merges with plaquettes). 
  
\end{corollary}

Knowing from remark \ref{fkpl} that $J_t^{(\mathcal{C},k)}$ is the unique solution to the Fokker-Planck equation associated with $\mathfrak{A}^{(\mathcal{C}_m,k)}$, the proof is essentially the same as in corollary \ref{pete}, after noting that  the first jump of the process occurs at a random exponential time of mean $( V^+ +V^- +2\frac{k}{d}(n-1)S)(\tilde{\mathcal{L}}')$ in order to determine the right integral term in the exponential.\\

\begin{remark}
One can show that the semigroups  $H_t^{(\mathcal{C})}$ and $J_t^{(\mathcal{C},k)}$ satisfy Harris recurrence conditions (Cf.  \cite{Revz}) on the set of connection moduli, which is compact. It follows that given any collection of geodesic loops $\tilde{\mathcal{L}}_0$, the integral of $\tau_A(\tilde{\mathcal{L}}_0)$ by the normalized Yang-Mills measure can be represented as $$\lim_{t\rightarrow \infty} \int J_t^{(\mathcal{C},k)}(I, dA)\tau_A((\tilde{\mathcal{L}}_0))=\lim_{t\rightarrow \infty}\mathbb{E}_{\tilde{\mathcal{L}}_0}^{SMD}[(-1)^{\bar{\mathfrak{m}}_t^+}e^{\int_0^t [(2k(n-1)-d+1)S+2V^-](\tilde{\mathcal{L}}_s)ds}]$$ in which we denote by $I$ modulus of the trivial connection.
\end{remark}


\begin{thebibliography}{11}
\bibitem{chater} S. Chatterjee. Rigorous solution of strongly coupled $SO(N)$ lattice gauge theory in the large $N$ limit. Comm. Mat. Phys. 366 (1) 203-268. (2019).
\bibitem{stfl} Y. Le Jan.  Markov paths, loops and fields.    \'{E}cole d'\'{E}t\'{e} de Probabilit\'{e}s de Saint-Flour XXXVIII (2008). Lecture Notes in Mathematics 2026.
Springer, Berlin-Heidelberg. (2011).
\bibitem{TL} Thierry Levy. Wilson loops in the light of spin networks. J. Geom.Phys. 52, 382-397 
(2004).
\bibitem{Lev2} Thierry Levy. The master field on the plane. Ast\'erisque 388 SMF (2017).
 \bibitem{pitpi}
J. Pitman. Poisson-Dirichlet and GEM invariant distributions for split-and merge transformations of an interval partition.
 Combinatorics, Probability and Computing 11. 501-514 (2002).
\bibitem{Proc} C.Procesi. The Invariant Theory of $n\times n$ matrices. Advances in Maths. 19 , 306-381 (1976).
\bibitem{Revz} D. Revuz  Markov chains, second edition. North Holland, Amsterdam. (1984).
\bibitem {Stebe}P.F. Stebe, A Residual Property of Certain Groups. Proc. American Math. Soc. 26, 37-42 (1970).
\end{thebibliography}
\end{document}